\documentclass[final,3p,times]{elsarticle}

\usepackage{booktabs}
 \usepackage{enumitem}
 \usepackage{bm}
\usepackage{lscape}
\usepackage{rotating}
 \usepackage{tikz}
\usetikzlibrary{shapes.geometric, arrows}
\usepackage{listings}
\usepackage{matlab-prettifier}

\usepackage{amsmath}
\usepackage{mathtools}
\mathtoolsset{showonlyrefs=true}

\usepackage{amssymb}
\usepackage{amsthm}
\usepackage{amsmath}
\newtheorem{lemma}{Lemma}%
\newtheorem{algorithm}{Algorithm}%

\journal{Nuclear Physics B}

\begin{document}

\begin{frontmatter}



\title{Method for Verifying Solutions of Sparse Linear Systems with General Coefficients}


\author[inst1]{Takeshi Terao}
\author[inst2]{Katsuhisa Ozaki}

\affiliation[inst1]{organization={Research Institute for Information Technology, Kyushu University},
            addressline={744 Motooka, Nishi-ku}, 
            city={Fukuoka},
            postcode={819-0395}, 
            state={Fukuoka},
            country={Japan}}
\affiliation[inst2]{organization={Department of Mathematical Science, Shibaura Institute of Technology},
            addressline={307 Fukasaku, Minuma-ku}, 
            city={Saitama},
            postcode={337-8570}, 
            state={Saitama},
            country={Japan}}

\begin{abstract}
This paper proposes a verification method for sparse linear systems $Ax=b$ with general and nonsingular coefficients.
A verification method produces the error bound for a given approximate solution.
Conventional methods use one of two approaches.
One approach is to verify the computed solution of the normal equation $A^TAx=A^Tb$ by exploiting symmetric and positive definiteness; however, the condition number of $A^TA$ is the square of that for $A$.
The other approach uses an approximate inverse matrix of the coefficient; however, the approximate inverse may be dense even if $A$ is sparse.
Here, we propose a method for the verification of solutions of sparse linear systems based on $LDL^T$ decomposition.
The proposed method can reduce the fill-in and is applicable to many problems.
Moreover, an efficient iterative refinement method is proposed for obtaining accurate solutions.

\end{abstract}


\end{frontmatter}


\section{Introduction}
\label{sec:sample1}

In this paper, we consider verified numerical computation for the linear system 
\begin{align}
    Ax=b,\label{eq:lin}
\end{align}
where $A$ is a real $n\times n$ matrix and $b$ is a real $n$-vector.
Solving linear systems is a basic task in scientific computing. 
When a linear system is solved using finite-precision floating-point arithmetic, the solution includes rounding errors.
Therefore, the accuracy of the computed result is unknown.
Verified numerical computations can be used to verify the accuracy of a computed result.
The process involves verifying that there is a real $n\times n$ matrix $Y$ that satisfies $\|YA-I\|_p=\alpha<1$.
If such $Y$ exists, 
\begin{align}
    \|A^{-1}\|_p\leq \frac{\|Y\|_p}{1-\alpha}\text{\quad and\quad}\|A^{-1}b-\widehat x\|_p\leq \|A^{-1}\|_p\cdot\|b-A\widehat x\|_p\leq \frac{\|Y\|_p\cdot\|b-A\widehat x\|_p}{1-\alpha}\label{eq:RA-I}
\end{align}
are satisfied for $L^p$-norms, where $\widehat x$ is an approximation of $x$, and $I$ is the identity matrix with an appropriate size.
Many verification methods based on \eqref{eq:RA-I} have been proposed~\cite{yamamoto1984error,rump1994verification,oishi2002fast,ogita2005largelinsys,ogita2005verified,ozaki2007method,ozaki2011algorithm,morikura2013verification,minamihata2020modified}.
Some of these methods are compared in reviews~\cite{ogita2009fast,rump_2010,rump2013accurate1,rump2013accurate2}.
In most cases, the matrix $Y$ that satisfies~\eqref{eq:RA-I} is dense, even if the given matrix $A$ is sparse.
Therefore, for large-scale systems, verification methods must be implemented on supercomputers due to the high required computation time and memory capacity~\cite{ozaki2021verified}.
To overcome this problem, verification methods for sparse linear systems with particular coefficients have been proposed~\cite{ogita2001fast,rump2007super,minamihata2013fast,minamihata2015improved,oishi2023lower}.
In particular, a verification method with a symmetric positive definite (SPD) coefficient~\cite{rump2007super} can be applied to the normal equation
\begin{align}
    A^TAx=A^Tb.
\end{align}
Although this method is simple and possibly fast, there are problems regarding its stability.
For example, the condition number of $A^T A$ is square of that of $A$, and there are rounding errors in the computation of $A^T A$ and $A^Tb$.
The verification of sparse linear systems with general coefficients is an open problem in {\it Grand Challenges and Scientific Standards in Interval Analysis}~\cite{neumaier2002grand}.
\begin{table}[htbp]\caption{Verification methods for sparse linear systems.}\label{tab:preworks}
    \begin{tabular*}{\textwidth}{@{\extracolsep\fill}lll}
        \toprule
        Author(s) & Year of publication & Note \\
        \midrule
        Rump~\cite{rump1994verification} & 1994 & $LDM^T$ decomposition based method\\
        Ogita et al.~\cite{ogita2006linsys} & 2005 & Memory reduction method\\
        Rump and Ogita~\cite{rump2007super} & 2007 &
        Method for SPD coefficients \\
        Minamihata et al.~\cite{minamihata2019} & 2019 & Fill-in reduction method\\
        \bottomrule
    \end{tabular*}
\end{table}

Some verification methods for sparse linear systems are listed in~Table\ref{tab:preworks}.
The method based on $LDM^T$ decomposition and the memory reduction is impractical for large-scale problems.
Although a method for normal equations can be used to verify the solution of a linear system, 
there is a limit on the condition number as $\kappa_2(A):=\|A\|\cdot\|A^{-1}\|\leq c_n\sqrt{u}^{-1}$, 
where $c_n$ is the positive constant corresponding to $n$, and $u$ is the unit roundoff (e.g., $u=2^{-53}$ for double-precision).
Although the fill-in reduction method runs to completion for $\kappa_2(A)\leq c_nu^{-1}$, the magnitude of fill-in is not sufficiently small and may not be negligible for large-scale matrices.

Here, we propose a method for verifying the solution of \eqref{eq:lin}.
This method involves the estimation of the upper bound of $\|A^{-1}\|$ for nonsingular $A$ and the computation of $\widehat x$ that satisfies a sufficiently small $\|b-A\widehat x\|$ with low computational cost.
The proposed method is applicable to problems with $\kappa_2(A)\leq c_n u^{-1}$.
Thus, even if the verification method for normal equations fails, the proposed method may still run to completion.
Moreover, the proposed method does not require the approximate inverse and the magnitude of fill-in is smaller than that in the fill-in reduction method~\cite{minamihata2019}.

Iterative refinement for linear systems using numerical computation was proposed by Moler~\cite{moler1967iterative}.
Its rounding error analysis has been reported in several studies~\cite{doi:10.1137/1.9781611971446,skeel1980iterative,doi:10.1137/1.9780898718027,doi:10.1137/1.9781611977523}.
A new error analysis and GMRES-based iterative refinement have been proposed~\cite{carson2017new,carson2018accelerating,higham2022mixed}.
Rump and Ogita proposed a rounding error analysis for iterative refinement using the factors of shifted SPD matrices~\cite{rump2007super}.
In the method proposed here, we apply matrix decomposition to shifted matrices obtained from the general coefficients.

The rest of this paper is organized as follows.
Section~\ref{sec:2} presents the notations used in this study and reviews previous work on the verification of sparse linear systems.
Section~\ref{sec:3} describes the proposed method.
Subsection~\ref{subsec:3-1} describes the proposed verification method for the lower bound of the minimum singular value $\underline \sigma_{\min}$ of a general matrix using shifted $LDL^T$ decomposition, and Subsection~\ref{subsec:3-2} describes the proposed iterative refinement method that uses the $LDL^T$ factors obtained when $\underline \sigma_{\min}$ is computed.
Section~\ref{sec:4} presents the computational results to demonstrate the performance and limitations of the proposed methods.
Section \ref{sec:5} gives the conclusions.

\section{Previous work}\label{sec:2}

Let $\mathbb{F}$ be a set of floating-point numbers, as defined according to the IEEE 754-2008 standard~\cite{4610935}.
Let $\mathbb{IF}$ be a set of intervals, as defined according to the INTLAB in MATLAB~\cite{rump1999intlab}.
In this paper, the bold notation $\mathbf{A}$ corresponding to $A$ denotes the interval that satisfies $A\in\mathbf{A}\in\mathbb{IF}^{n\times n}$.
The notations $\mathit{fl}(\cdot), \mathit{fl}_{\bigtriangleup}(\cdot)$, and $\mathit{fl}_{\bigtriangledown}(\cdot)$ denote the numerical results obtained with roundTiesToEven, roundTowardPositive, and roundTowardNegative in IEEE 754-2008, respectively.
Throughout this paper, we assume no pivoting in the matrix decomposition.
For example, $LU$ decomposition for a sparse matrix often uses the minimum degree ordering such that $PAQ=LU$ to reduce the cost due to the fill-in, where $P$ and $Q$ are permutation matrices.
Matrices $P$ and $Q$ for $LU$ and $LDL^T$ decomposition are omitted here to simplify the description of inequalities and algorithms.
For matrices and vectors, $\|\cdot \|$ and $\|\cdot\|_p$ denote the $L^2$-norm and the $L^p$-norm, respectively. 
The inequality $A<B$ for matrices $A$ and $B$ indicates $a_{ij}<b_{ij}$ for all $i,j$ pairs.
This can also be used for vectors.

\subsection{Using Cholesky decomposition}

Let $A\in\mathbb{F}^{n\times n}$ and $0\leq \alpha\in\mathbb{F}$ be given.
On the INTLAB in MATLAB,
the interval $\mathbf{S}\ni A^TA-\alpha I$ can be obtained using interval matrix multiplication and subtraction. 
The function {\tt isspd}~\cite{rump2006verification} in INTLAB can verify the positive definiteness for the given $\mathbf{S}\in\mathbb{IF}^{n\times n}$
(i.e., all symmetric matrices $S\in\mathbf{S}$ is SPD).

If $\forall S$ is positive definite, then $\|A^{-1}\|\leq \sqrt{\alpha}^{-1}$ is satisfied.

\begin{algorithm}[Cholesky decomposition based algorithm]\label{alg:chol}
Let $A\in\mathbb{F}^{n\times n}$ and $\alpha>0$ be given.
This algorithm computes the lower bound of the minimum singular value for a general matrix.
\begin{enumerate}[label=\textbf{Step \arabic*.},itemsep=1ex, leftmargin=1.5cm]
    \item Compute $\mathbf{S}\ni A^TA-\alpha I\in\mathbb{IF}^{n\times n}$ using interval arithmetic.
    \item Run the function {\rm isspd} in the INTLAB for $\mathbf{S}$.
    If {\rm isspd} runs to completion, return the lower bound of the minimum singular value $\sqrt{\alpha}$; 
    otherwise, the verification fails.
\end{enumerate}
\end{algorithm}

\subsection{Using $LU$ decomposition}

Assume that the computed factors of $LU$ decomposition $\widehat L$ and $\widehat U$ are obtained such that $A\approx \widehat L\widehat U$, and denote the approximate inverse matrices of $\widehat L$ and $\widehat U$ as $X_L$ and $X_U$, respectively.
For $H$, which denotes the comparison matrix of $X_LAX_U$, define $H=D-E$ from the diagonal part $D$ and the off-diagonal part $-E$ of $H$.
If there exists $v>\bm{0}$ that satisfies $u=Hv>\bm{0}$, define
\begin{align}
    w_k=\max_{1\leq i\leq n}\frac{G_{ik}}{u_i}\text{\quad for\quad}1\leq k\leq n,
\end{align}
where $G=ED^{-1}\geq O$.
Then, from~\cite{rump2013accurate2}, $A$ is nonsingular and
\begin{align}
    |(X_LAX_U)^{-1}|\leq D^{-1}+vw^T\\
    |A^{-1}b-\widehat x|\leq |X_U|(D^{-1}+vw^T)|X_L(b-A\widehat y)|.
\end{align}
Then, if $\|X_LAX_U-I\|_p<1$, it holds that
\begin{align}
    \|A^{-1}\|_p\leq \frac{\|X_U\|_p\cdot\|X_L\|_p}{1-\|X_LAX_U-I\|_p} \text{\quad and \quad} \|A^{-1}b-\widehat x\|_p\leq \frac{\|X_U\|_p\cdot\|X_L(b-A\widehat x)\|_p}{1-\|X_LAX_U-I\|_p}.\label{eq:minamihata}
\end{align}
This method aims to reduce the fill-in of $X_LAX_U$ compared to that of $X_UX_LA$ from~\cite{minamihata2019}.
Although this method achieves high performance if the magnitude of fill-in is sufficiently small, the fill-in in the computation of $X_L$ and $X_U$ may not be negligible for large-scale matrices. 
For the implementation, the approximate minimum degree ordering algorithm should be used~\cite{amestoy2004algorithm}.

\begin{algorithm}[$LU$ decomposition based algorithm]\label{alg:lu}
Let $A\in\mathbb{IF}^{n\times n}$ be given.
\begin{enumerate}[label=\textbf{Step \arabic*.},itemsep=1ex, leftmargin=1.5cm]
    \item Decompose $A$ into the approximate product $A\approx \widehat L\widehat U$, approximately.
    \item Compute the approximate inverse $X_L$ and $X_U$ of $\widehat L$ and $\widehat U$, respectively.
    \item Compute $\rho\geq \|I-X_LAX_U\|_p$ using interval arithmetic.
    If $\rho<1$, return the upper bound of $\|X_L\|_p\cdot\|X_U\|_p/(1-\rho)$, as the upper bound of $\|A^{-1}\|_p$; otherwise the verification fails.
\end{enumerate}
\end{algorithm}

\subsection{Modified $LU$ based method}

There are cases in which the magnitude of fill-in in $X_LAX_U$ cannot be neglected.
Ogita and Oishi proposed a fast method for dense linear systems~\cite{ogita2005largelinsys}.
Here, we combine the techniques in \cite{minamihata2019} and \cite{ogita2005largelinsys}.
We have
\begin{align}
    X_LAX_U&=I+(X_LA-\widehat U) X_U+(\widehat U X_U-I)\\
    &=I+\Delta,\quad |\Delta|\leq  |X_LA-\widehat U| |X_U|+|\widehat U X_U-I|.\label{eq:reduce}
\end{align}
Under the assumption that $F=X_LA$ and $G=FX_U$,  Algorithm~\ref{alg:lu} requires storing $F$ and $G$ at the same time.
From~\eqref{eq:reduce}, if we obtain the upper bound of $\|F-\widehat U\|_p$, we can clear $X_L$ and $F$.
Therefore, we can reduce the amount of memory required to compute the inclusion of $X_LAX_U$.
Note that the method based on~\eqref{eq:reduce} may have weaker stability than that of the method based on \eqref{eq:minamihata}.

\begin{algorithm}[Modified $LU$ decomposition based algorithm]\label{alg:mlu}
Let $A\in\mathbb{IF}^{n\times n}$ be given.
\begin{enumerate}[label=\textbf{Step \arabic*.},itemsep=1ex, leftmargin=1.5cm]
    \item Decompose $A$ into the approximate product $A\approx \widehat L\widehat U$.
    \item Compute the approximate inverses $X_L$ and $X_U$ of $\widehat L$ and $\widehat U$, respectively.
    \item Compute the upper bound of $\|X_LA-\widehat U\|_p\cdot\|X_U\|_p=\alpha$.
    If $\alpha>1$, switch to Algorithm~\ref{alg:lu} using the results $X_LA$ and $X_U$.
    \item Compute $\|\widehat UX_U-I\|_p\leq\beta$ using interval arithmetic.
    If $\alpha+\beta<1$, return $\|X_L\|_p\cdot\|X_U\|_p/(1-\alpha-\beta)$; otherwise, the verification fails or switch to Algorithm~\ref{alg:lu} using the results $X_LA$ and $X_U$.
\end{enumerate}
\end{algorithm}

\section{Proposed method}\label{sec:3}

For the approximate solution $\widehat x$ of \eqref{eq:lin} with the nonsingular coefficient $A$, it holds
\begin{align}
    \|A^{-1}b-\widehat x\|\leq \|A^{-1}\|\cdot\|b-A\widehat x\|.
\end{align}
Here, the upper bound of $\|b-A\widehat x\|$ can be computed using $\mathcal{O}(n^2)$ floating-point operations.
Thus, if the upper bound of $\|A^{-1}\|$ is obtained, the norm of the residual vector can quickly be bounded. 
In this section, we first propose a verification method for the upper bound of $\|A^{-1}\|$.
Next, we propose an iterative refinement method for computing a sufficiently accurate solution of $\eqref{eq:lin}$.

\subsection{Verification of spectral norm of inverse matrix}\label{subsec:3-1}

Let $A\in\mathbb{R}^{n\times n}$ be given and denote singular value decomposition as
\begin{align}
    A=U\Sigma V^T,\quad U^TU=V^TV=I,\label{eq:svd}
\end{align}
where $U$ and $V$ are orthogonal matrices and $\Sigma$ is a diagonal matrix, where $\Sigma=\mathrm{diag}(\sigma_{1},\dots,\sigma_{n})$.
Note that the singular value decomposition is introduced here to facilitate explanation; it is not required in the proposed method.
For nonsingular~$A$, we have $\|A^{-1}\|=\min_{1\leq i\leq n}(\sigma_i)^{-1}=:\sigma_{\min}^{-1}$.
Thus, we consider the verification of the lower bound of $\sigma_{\min}$.
From~\eqref{eq:svd}, the matrices are defined using block representations as 
\begin{align}
    \bar A=
    \left(
    \begin{array}{cc}
        O & A^T \\
        A & O
    \end{array}
    \right),\quad
    \bar Q=
    \frac{1}{\sqrt{2}}\left(
    \begin{array}{cc}
        V & V \\
        U & -U
    \end{array}
    \right),\quad
    \bar \Sigma=
    \left(
    \begin{array}{cc}
        \Sigma & O \\
        O & -\Sigma
    \end{array}
    \right).\label{eq:ex}
\end{align}
It is known that
\begin{align}
    \bar A=\bar Q\bar \Sigma \bar Q^T,\quad \bar Q^T\bar Q=I.
\end{align}
From this, for $\Lambda(\bar A)$ denoting the spectral of $\bar A$, we have
\begin{align}
    \Lambda(\bar A)=\{-\sigma_1,\dots,-\sigma_n,\sigma_n,\dots,\sigma_1\}.
\end{align}
Therefore, for a real scalar $\theta$, it holds that 
\begin{align}
    \Lambda(\bar A+\theta I)=\{-\sigma_{1}+\theta,\dots,-\sigma_n+\theta,\sigma_n+\theta,\dots,\sigma_1+\theta\}.\label{eql:spectral_sA}
\end{align}
If the number of positive (or negative) eigenvalues is equal to $n$ for $\bar A+\theta I$, it holds that $A$ is nonsingular and $\sigma_{\min}(A)>|\theta|$.
For counting the number of positive eigenvalues, the $LDL^T$ decomposition $\bar A+\theta I=LDL^T$ is often used, where $L\in\mathbb{R}^{n\times n}$ is unit lower triangular and $D\in\mathbb{R}^{n\times n}$ is block diagonal with diagonal blocks of dimension one or two \cite{bunch1977some,ashcraft1998accurate}.
From Sylvester's law of inertia, the inertia of $\bar A+\theta I$ and $D$ are the same.
Based on this, we provide the following theorem.

\begin{lemma}\label{lem:npe}
    Let $\bar A\in\mathbb{R}^{2n\times 2n}, \theta\in\mathbb{R}$, and $\widehat L,\widehat D\in\mathbb{R}^{2n\times 2n}$ be given.
    If the number of positive (or negative) eigenvalues for $\widehat D$ is equal to $n$,
    it holds that
    \begin{align}
        \sigma_{\min}(A)\geq|\theta|-\rho,\label{eq:lem1-1}
    \end{align}
    where $\rho\geq \|\bar A+\theta I-\widehat L\widehat D\widehat L^T\|$.
    In particular, when $|\theta|>\rho$, matrix $A$ is nonsingular and
    \begin{align}
        \|A^{-1}\|\leq \frac{1}{|\theta|-\rho}\label{eq:lem1-2}
    \end{align}
    is satisfied.
\end{lemma}

\begin{proof}
    Suppose that
    \begin{align}
        \bar A+\theta I= \widehat L\widehat D\widehat L^T+\Delta, \quad \rho\geq \|\Delta\|,
    \end{align}
    where matrix $\Delta$ is symmetric.
    Then, we have
    \begin{align}
        \lambda_i(\bar A)+\theta-\rho\leq \lambda_i(\widehat L\widehat D\widehat L^T)\leq \lambda_i(\bar A)+\theta+\rho, \quad \lambda_{i}(\cdot)\in\Lambda(\cdot)
    \end{align}
    for all $i$.
    From the assumption that the number of positive (or negative) eigenvalues is $n$, $\sigma_{\min}(A)\geq|\theta|-\rho$ is satisfied from \eqref{eql:spectral_sA}.
    If $\sigma_{\min}> 0$, $A$ is nonsingular and $\|A^{-1}\|=\sigma_{\min}(A)^{-1}$. 
    Thus, \eqref{eq:lem1-2} holds.
\end{proof}

Based on Lemma~\ref{lem:npe}, our algorithm is as follows.

\begin{algorithm}[Proposed method based on $LDL^T$ decomposition]\label{alg:proposed}\ 

\begin{enumerate}[label=\textbf{Step \arabic*.},itemsep=1ex, leftmargin=1.5cm]
    \item Define $\bar A=\left(\begin{array}{cc}
        O & A^T \\
        A & O
    \end{array}\right)$ and $\theta$ from a given matrix $A$.
    \item Decompose $\bar A+\theta I$ into a product such that $\bar A+\theta I\approx \widehat L\widehat D\widehat L^T$.
    \item Count the number of positive (or negative) eigenvalues of $\widehat D$. 
        If $\mathrm{npe}(\widehat D)\not=\mathrm{nne}(\widehat D)$, restart from Step 2 using smaller $\theta$ or the verification fails.
    \item Compute the upper bound of residual norm $\rho \geq \|A- \widehat L\widehat D\widehat L^T\|$. If $|\theta|>\rho$, return the lower bound of the minimum singular value $|\theta|-\rho$ of $A$; otherwise, restart from Step 2 using larger $|\theta|>\rho$ or the verification fails.
\end{enumerate}
\end{algorithm}

In Step 1, the shift amount $\theta$ can be computed in MATLAB as $\alpha*\mathtt{eigs}(\bar A, 1, 'smallest')$ where $0<\alpha<1$.
In Steps 3-4, ``verification fails" indicates that there is a possibility that $A$ is singular or ill-conditioned.
The main computational cost is $LDL^T$ decomposition and estimation of the upper bound of $\|A-\widehat L\widehat D\widehat L^T\|$.

Rump~\cite{rump_2010} stated that ``{\it There are some methods based on an $LDL^T$ decomposition discussed in~\cite{rump1994verification}; however, the method is only stable with some pivoting, which may, in turn, produce significant fill-in"}. 
Thus, we recommend preconditioning to improve the speed and stability.
However, because the sparsity of the preconditioned matrix is worse than that of the original matrix, preconditioning using incomplete $LU$ or incomplete Cholesky decomposition should not be applied to the proposed method.
One of the most useful preconditioning algorithms is an algorithm for permuting large entries to the diagonal of a sparse matrix~\cite{duff2001algorithms}.
Using this algorithm, we can obtain matrices $P$, $R$, and $C\in\mathbb{R}^{n\times n}$, where $R$ and $C$ are diagonal matrices and $P$ is the permutation matrix. 
Thus, matrix $\widetilde A=RPAC$ permutes and rescales matrix $A$ such that the new matrix $\widetilde A$ has a diagonal with entries of magnitude 1, and its off-diagonal entries have a magnitude not greater than 1.
Through this, we have
\begin{align}
    \sigma_{\min}(A)\geq \frac{\sigma_{\min}(\widetilde A)}{\|R\|\cdot\|C\|}=\frac{\sigma_{\min}(\widetilde A)}{\displaystyle\max_i (|r_{ii}|)\cdot \max_i (|c_{ii}|)}.
\end{align}
This function has been implemented in MATLAB as ``equilibrate".
Here, there is a possibility that $\widetilde A\notin\mathbb{F}^{n\times n}$.
Interval arithmetic using INTLAB is useful such as $$\widetilde{\mathbf{A}}=R*P*\mathtt{intval}(A)*C.$$
In this case, our method computes the $LDL^T$ decomposition of a matrix from $\mathrm{mid}(\widetilde{\mathbf{A}})$ rather than from interval $\widetilde{\mathbf{A}}$ itself. 
Then, the proposed method computes the shifted $LDL^T$ decomposition.
To avoid rounding errors, suppose that 
\begin{align}
    \widetilde r_{ii}=\mathrm{sign}(r_{ii})\cdot 2^{[\log_2|r_{ii}|]},\quad \widetilde c_{ii}=\mathrm{sign}(c_{ii})\cdot 2^{[\log_2|c_{ii}|]},
\end{align}
where $[\alpha ]=\{z \in \mathbb{Z}\ |\ |z-\alpha|\leq 0.5\}$.
Then, it holds that $\widetilde RPA\widetilde C\in\mathbb{F}^{n\times n}$ under the assumption that underflow does not occur.

\subsection{Iterative refinement}
\label{subsec:3-2}
There are several approaches for the iterative refinement of the approximate solution $\widehat x^{(0)}$ of the linear system $Ax=b$.
Assume that $A\approx M\approx M_LM_U$, where $M_L,M_U\in\mathbb{R}^{n\times n}$.
One approach is to compute
\begin{align}
    A^{-1}b-\widehat x\approx M^{-1}(b-A\widehat x^{(k)})\approx \widehat e
\end{align}
and update the solution as $\widehat x^{(k+1)}\leftarrow \widehat x^{(k)}+\widehat e^{(k)}$.
There are several methods based on this approach~\cite{moler1967iterative,ogita2003fast}.
Another approach is to solve the linear system
\begin{align}
    Ae=b-A\widehat x
\end{align}
using Krylov subspace methods with the preconditioner $M$ or $M_LM_U$~\cite{carson2017new,carson2018accelerating}.

Our method computes the $LDL^T$ decomposition of a shifted matrix $\bar A-\theta I$ rather than $\bar A$ itself. 
If the proposed verification runs to completion, it implies that the lower bound on the smallest singular value of $A$ was successfully obtained.
Therefore, the nonsingularity of $A$ is proved.
From Lemma~\ref{lem:npe}, if the lower bound of the minimum singular value of $A$ is obtained as $\delta:=|\theta|-\rho>0$,
it holds that
\begin{align}
    \|A^{-1}b-\widehat x\|\leq\|A^{-1}\|\cdot\|b-A\widehat x\|\leq \frac{\|b-A\widehat x\|}{\delta}=:\epsilon
\end{align}
for the linear system $Ax=b$ and its approximate solution $\widehat x$.
However, the error bound becomes overestimated in many cases.
Therefore, we propose an efficient iterative method for refining the accuracy such that $\widehat x$ satisfies a sufficiently small $\max_{1\leq i\leq n}(\epsilon/\widehat x_i)$.

From the proposed verification method for obtaining $\delta$, we have the approximate factors $\widehat L$ and $\widehat D$ such that $\bar A+\theta I\approx\widehat L\widehat D\widehat L^T$.
For the matrix $\bar A$ and the linear system $Ax=b$, it is shown that
\begin{align}
    \left(
    \begin{array}{cc}
        O & A^T \\
        A & O
    \end{array}
    \right)
    \left(
    \begin{array}{c}
        x \\
        Ax
    \end{array}
    \right)
    =
    \left(
    \begin{array}{c}
        A^Tb \\
        b
    \end{array}
    \right).
\end{align}
Here, when $\theta$ is sufficiently small, we can assume that $\bar A\approx \widehat L\widehat D\widehat L^T$.
Then, for $e\in\mathbb{R}^{n}$, with the error of $\widehat x$ denoted as $x=\widehat x+e$, it holds that 
\begin{align}
    \left(
    \begin{array}{c}
        e \\
        Ae
    \end{array}
    \right)
    &=
    \left(
    \begin{array}{cc}
        O & A^T \\
        A & O
    \end{array}
    \right)^{-1}
    \left(
    \left(
    \begin{array}{c}
        A^Tb \\
        b
    \end{array}
    \right)
    -
    \left(
    \begin{array}{cc}
        O & A^T \\
        A & O
    \end{array}
    \right)
    \left(
    \begin{array}{c}
        \widehat x \\
        A\widehat x
    \end{array}
    \right)
    \right)\\
    &=
    \left(
    \begin{array}{cc}
        O & A^T \\
        A & O
    \end{array}
    \right)^{-1}
    \left(
    \left(
    \begin{array}{c}
        A^Tb \\
        b
    \end{array}
    \right)
    -
    \left(
    \begin{array}{c}
        A^TA\widehat x \\
        A\widehat x
    \end{array}
    \right)
    \right)\\
    &\approx
    (\widehat L \widehat D\widehat L^T)^{-1}
    \left(
    \begin{array}{c}
        A^Tr \\
        r
    \end{array}
    \right),\quad r=b-A\widehat x
    .
\end{align}
Therefore, the approximate correction $\widehat e$ can be obtained from matrix-vector products and triangular systems.
Then, the approximate solution $\widehat x$ can be updated as $\widehat x\leftarrow \widehat x+\widehat e$.
This process is repeated until a sufficiently small residual is obtained.

Based on the above, the proposed algorithm is as follows.
\begin{algorithm}[Iterative refinement]\label{alg:iterref}
Let $A\in\mathbb{F}^{n\times n}$ and $\widehat L,\widehat D\in\mathbb{F}^{2n\times 2n}$ be given.
Let the approximate solution $\widehat x$ of the linear system $Ax=b$ also be given.
\begin{enumerate}[label=\textbf{Step \arabic*.},itemsep=1ex, leftmargin=1.5cm]
    \item Compute $\widehat r\leftarrow b-A\widehat x$ with highly accurate arithmetic and round $\widehat r$ to the working precision.
    \item Define $\bar r=\left(\begin{array}{c}
        A^T\widehat r\\
        \widehat r
    \end{array}\right)$ and compute $\bar e\leftarrow \widehat L^{-T}(\widehat D^{-1}(\widehat L^{-1}\bar r))$.
    \item Update $\widehat x_i\leftarrow \widehat x_i+\bar e_i$ for $1\leq i\leq n$.
    If $\widehat x$ is sufficiently accurate, return the approximate solution $\widehat x$; 
    otherwise, restart from Step 1.
\end{enumerate}

\end{algorithm}

From~Algorithm~\ref{alg:iterref}, define $\widehat x^{(0)}\leftarrow\widehat L^{-T}\widehat D^{-1}\widehat L^{-1}b$ and $\widehat x^{(k+1)}_i\leftarrow \widehat x^{(k)}_i+\widehat e^{(k)}_i$ for $k>0$ and $1\leq i\leq n$.
Then, for $\alpha=\|\theta A^{-1}\|,$ we have
\begin{align}
    \|\widehat x^{(k+1)}-A^{-1}b\|\leq \left(\frac{\alpha}{1-\alpha}\right)^{k+1}\|\widehat x^{(0)}-A^{-1}b\|
\end{align}
from~\cite{rump2007super}.
Therefore, if we set a sufficiently small $\theta$ compared to $\sigma_{\min}(A)$, an accurate solution can be obtained using Algorithm~\ref{alg:iterref} with low computational cost.
The cost of one iteration is at most $\mathcal{O}(n^2)$ floating-point operations.
Because the cost of $LDL^T$ decomposition is about $\mathcal{O}(n^3)$, the cost of the iterative refinement is negligible.
Assume that the computed solution is divided as $\widehat x=y+z$, where $ y, z\in\mathbb{F}^n$ are obtained through the iterative refinement.
Then, 
\begin{align}
    | z|-\delta^{-1}\cdot\|b-A\widehat x\|e\leq|A^{-1}b- y|\leq | z|+\delta^{-1}\cdot\|b-A\widehat x\|e
\end{align}
is satisfied from~\cite{ogita2003fast,ogita2009tight}.

\section{Numerical experiments}\label{sec:4}

We conducted numerical experiments on the proposed method and compared the results with those obtained using existing methods.
The experiments were conducted on a computer with 
an Intel\textregistered Core\texttrademark{} i7-1270P CPU at 2.20 GHz and 32.0 GB of RAM running Windows 11 Pro.
MATLAB 2023a and INTLAB Ver. 12 were used.
The methods were implemented in MATLAB without the MEX (MATLAB Executable, using C or FORTRAN) function.
The code for the proposed method is given in~\ref{sec:MATLABcodes}.

The options \texttt{precond} and \texttt{acc} 
are specified as 
\begin{align}
    \mathtt{verify\_sigmin(A, precond, acc)}.
\end{align}
$\mathtt{precond}=0$ and $1$ indicate no preconditioning and preconditioning, respectively.
$\mathtt{acc}=0$ and $1$ indicate fast interval matrix multiplication and accurate interval matrix multiplication, respectively.
In other words, in the original MATLAB code~\ref{code:sigmin}, we set
$
    \mathtt{intvalinit ( ' SharpIVmult ') ;} \ \mathtt{and}\ \mathtt{intvalinit ( ' FastIVmult ') ;},
$
when \texttt{acc=1} and \texttt{acc=0}, respectively.

We denote the proposed method as P(\texttt{precond}, \texttt{acc}).
Table~\ref{tab:time} shows the elapsed times for the proposed method and approximate $LU$ decomposition obtained in MATLAB with the default values \texttt{[L,U,P,Q]=lu(A)}. 
The amount of shift $\theta$ in Algorithm~\ref{alg:proposed} was computed using eigensolver in MATLAB with the scaling $$\mathtt{theta = 0.5*abs(eigs(G,1,'smallestabs'))}.$$
The results show that the proposed method can verify the lower bound of the minimum singular value of sparse matrices $A$.
Moreover, preconditioning is useful for improving stability and computation time, allowing the proposed method to be applied to many problems.

\begin{table}[htbp]\caption{Comparison of elapsed time for the proposed method and approximate $LU$ decomposition for $A$. The notation nnz$(A)$ denotes the number of nonzero entries of $A$.
The notation ``--" denotes that verification failed.}\label{tab:time}
    \begin{tabular*}{\textwidth}{@{\extracolsep\fill}lrrrrrrr}
        \toprule
         &  & &  \multicolumn{5}{c}{Elapsed time [seconds]}\\\cmidrule{4-8}
        Name & $n$ & nnz$(A)$ & $LU$ & P(0,0) & P(0,1) & P(1,0) & P(1,1)\\\hline
        fd12 & 7,500 & 28,462 & 0.0205 & 0.3524 & 0.3472 & 0.3009 & 0.3040\\
        fd15 & 11,532 & 44,206 & 0.0401 & 0.9446 & 0.9442 & 0.6634 & 0.6753\\
        fd18 & 16,428 & 63,406 & 0.0649 & 2.0265 & 2.0554 & 1.1946 & 1.1924\\\midrule
        cz5108 & 5,108 & 51,412 & 0.0121 & 0.1498 & 0.1508 & 0.1479 & 0.1519\\
        cz10228 & 10,228 & 102,876 & 0.0244 & 0.3669 & 0.3570 & 0.3647 & 0.3699\\
        cz20468 & 20,468 & 206,076 & 0.0481 & 0.8241 & 0.8062 & 0.7338 & 0.7431\\
        cz40948 & 40,948 & 412,148 & 0.0928 & -- & -- & 1.7373 & 1.7403\\\midrule
        epb1 & 14,734 & 95,053 & 0.0423 & 3.3275 & 3.2686 & 3.3212 & 3.3149\\
        epb2 & 25,228 & 175,027 & 0.1459 & 3.5979 & 3.5968 & 3.8957 & 3.9077\\
        epb3 & 84,617 & 463,625 & 0.4137 & 7.5639 & 7.5647 & 7.6185 & 7.6806\\\midrule
        Goodwin\_023 & 6,005 & 182,168 & 0.0364 & 0.6672 & 0.6518 & 0.6954 & 0.7251\\
        Goodwin\_030 & 10,142 & 312,814 & 0.0724 & 1.4230 & 1.4215 & 1.4007 & 1.4008\\
        Goodwin\_040 & 17,922 & 561,677 & 0.1564 & 3.3993 & 3.3707 & 3.0469 & 3.0471\\
        Goodwin\_054 & 32,510 & 1,030,878 & 0.3182 & 13.4612 & 13.4914 & 6.9202 & 6.9572\\
        Goodwin\_071 & 56,021 & 1,797,934 & 0.6574 & 36.1227 & 36.2909 & 15.2437 & 15.2072\\
        Goodwin\_095 & 100,037 & 3,226,066 & 1.3254 & -- & -- & 35.0571 & 35.1644\\
        Goodwin\_127.mat & 178437 & 5778545 & 3.1032 & -- & -- & 81.9969 & 82.1649\\\midrule
        ASIC\_100k & 99,340 & 940,621 & 0.5673 & 23.5867 & 23.9448 & 15.2047 & 15.1801\\
        ASIC\_320k & 321,821 & 1,931,828 & 1.9481 & -- & -- & 122.0162 & 123.3035\\
        ASIC\_680k & 682,862 & 2,638,997 & 0.6922 & -- & --& 251.8666 & 251.6116\\        \bottomrule
    \end{tabular*}
\end{table}

Table~\ref{tab:ratio} shows the ratio of the computation time of the functions in the proposed method P(1,1).
The notations used in the table are as follows:
\begin{description}[labelwidth=7em]
   \item[\texttt{equilibrate}]: time required for preconditioning
   \item[\texttt{eigs}]: time required for computing the amount of shift using eigensolver.
   \item[\texttt{ldl}]: time required for $LDL^T$ decomposition
   \item[Mul]: time required for interval matrix multiplication for computing the upper bound of the residual norm
\end{description}
In many cases, the main computational cost is $LDL^T$ decomposition and matrix multiplications for computing the residual.
However, in some cases (e.g., ASIC\_680k),
preconditioning (\texttt{equilibrate}) is the main computational cost.
Here, if the computational cost of \texttt{equilibrate} is not negligible, the proposed method without preconditioning may fail (cf. Table~\ref{tab:time}).

\begin{table}[htbp]\caption{Ratio of elapsed time of functions in proposed method P(1,1).}\label{tab:ratio}
    \begin{tabular*}{\textwidth}{@{\extracolsep\fill}lrrrrrr}
        \toprule
             & &   \multicolumn{5}{c}{Ratio}\\\cmidrule{3-7}
        Name & condest$(A)$ & \texttt{equilibrate} & \texttt{eigs} & \texttt{ldl} & Mul. & others\\\hline
        fd18 & 1.23e+14 & 6.8 & 14.5 & 29.1 & 46.1 & 3.5\\
        cz40948 & 9.02e+10 & 1.0 & 28.2 & 32.4 & 32.7 & 5.7\\
        epb3 & 7.45e+04 & 1.2 & 14.4 & 46.8 & 35.8 & 1.8\\
        Goodwin\_127 & 8.19e+07 & 14.4 & 9.9 & 25.9 & 48.9 & 0.9\\
        ASIC\_680k & 9.47e+19 & 82.5 & 7.1 & 3.0 & 6.9 & 0.4\\
        \bottomrule
    \end{tabular*}
\end{table}

Next, we compare the performance of the proposed method to that of existing methods.
We set $\alpha =0$ in Algorithm~\ref{alg:chol}.
Note that this is the verification of the positive definiteness of the interval $\bm G\ni A^TA$.
When the given matrix is not ill-conditioned, and the fill-in of $A^TA$ from $A$ is small, Algorithm~\ref{alg:chol} is an efficient method (e.g., epb3).
When the given matrix is not large-scale, Algorithm~\ref{alg:lu} may have high stability because the approximate inverses of $LU$ factors are obtained.
In addition, Algorithm~\ref{alg:lu} may become more efficient if it is implemented on computers with a lot of memory.
In contrast, the proposed method is well-balanced in terms of speed, stability, and required memory.

\begin{table}[htbp]\caption{Elapsed times of previous works Algorithm~\ref{alg:chol} with $\alpha =0$ and Algorithm~\ref{alg:lu}, and proposed method P(1,1). 
\texttt{intval(A)$\backslash$b} is an INTLAB function for computing the enclosure of the solution of a linear system $Ax=b$ with a dense coefficient. 
The notation ``--" denotes that verification failed and ``OOM" denotes out of memory.}\label{tab:compare}
    \begin{tabular*}{\textwidth}{@{\extracolsep\fill}lrrrrr}
        \toprule
             &    \multicolumn{5}{c}{Elapsed time [sec]}\\\cmidrule{2-6}
        Name & \texttt{intval(A)$\backslash$b} & Algorithm~\ref{alg:chol} & Algorithm~\ref{alg:lu} & Algorithm~\ref{alg:mlu} & P(1,1)\\\hline
        fd18 & 185.1658 & -- & 35.2352 & 30.9643 & 0.9159\\
        cz40948 & OOM & -- & 3886.1965 & 996.8880 & 1.7403\\
        epb3 & OOM &  0.6937 & 1618.4312 & 889.2158 & 6.5058\\
        Goodwin\_127 & OOM & -- & OOM & OOM & 65.8549\\
        ASIC\_680k & OOM & OOM & OOM & OOM & 250.5345\\
        \bottomrule
    \end{tabular*}
\end{table}

Finally, we evaluated the performance of iterative refinement.
For highly accurate arithmetic, we used the Multiprecision Computing Toolbox, a MATLAB extension for computing with arbitrary precision~\cite{advanpix2006multiprecision}.
The right-hand side vector was generated as $$\mathtt{b=A*v}$$ with double-precision arithmetic, where $\mathtt{v}=(1,\dots, 1)^T$.
Table~\ref{tab:iterref} shows the elapsed time, number of iterations, and accuracy of the solution.
Let the approximate solution $\widehat x=\widehat x_1+\widehat x_2$ be given.
Then, assume that the interval solution $\mathbf{x}$ is given such that
\begin{align}
    \mathrm{mid}(\mathbf{x})=\widehat x_1\text{\quad and \quad} \mathrm{rad}(\mathbf{x})=\mathit{fl}_{\bigtriangleup}(|\widehat x_2|+\delta^{-1}\cdot\sup\|b-A\widehat x\|\cdot v).
\end{align}
Suppose that the tolerance
\begin{align}
    \frac{\|b-A\widehat x\|}{\|b\|}\leq 2^{-53}\cdot \delta\label{eq:cndeq}
\end{align}
is used in the solution process, where $\delta$ is the lower bound of the minimum singular value computed by the proposed method P(1,1).
From Table~\ref{tab:iterref}, the iterative refinement has a low computational cost compared to that of the verification of the minimum singular value.
If the shift amount in the method P(1,1) decreases, the number of iterations can be reduced.
However, note that the possibility that the verification will fail increases with decreasing shift amount.
In addition, we can obtain the solution with higher accuracy using smaller $\delta$ in~\eqref{eq:cndeq}. 

\begin{table}[htbp]\caption{Elapsed time and number of iterations for Algorithm~\ref{alg:iterref}}\label{tab:iterref}
    \begin{tabular*}{\textwidth}{@{\extracolsep\fill}lrrrr}
        \toprule
         &  \multicolumn{2}{c}{Elapsed time [sec]}\\\cmidrule{2-3}
        Name & P(1,1) & Iter. Ref. & No. of Iter. & $\max|\mathrm{rad}(\mathbf{x}_i)/\mathrm{mid}(\mathbf{x}_i)|$\\\hline
        fd18 & 1.2159 & 0.6146 & 35 & 1.4776e-11 \\
        cz40948 & 1.7403 & 1.0305 & 41 & 7.7716e-16 \\
        epb3 & 6.5058 & 2.7515 & 36 & 1.0991e-14\\
        Goodwin\_127 & 65.8549 & 9.2349 & 25 & 1.4732e-12\\
        ASIC\_680k & 250.5345 & 12.5646 & 31 & 1.0198e-08\\
        \bottomrule
    \end{tabular*}
\end{table}

If the computational cost of $LU$ decomposition is low, then 
the time required for iterative refinement can be shortened.
Algorithm~\ref{alg:iterref} uses the $LDL^T$ factors of the shifted matrix.
Many iterations are thus required for the refinement.
However, if we can use the $LU$ factors of the original matrix, the size of the matrix will be halved, decreasing the number of iterations.
Table~\ref{tab:iterref_LU} shows the results of iterative refinement using $LU$ decomposition.
The elapsed time of the iterative refinement includes that of $LU$ decomposition.


\begin{table}[htbp]\caption{Elapsed time and number of iterations using $LU$ decomposition.}\label{tab:iterref_LU}
    \begin{tabular*}{\textwidth}{@{\extracolsep\fill}lrrrr}
        \toprule
         &  \multicolumn{2}{c}{Elapsed time [sec]}\\\cmidrule{2-3}
        Name & P(1,1) & Iter. Ref. & No. of Iter. & $\max|\mathrm{rad}(\mathbf{x}_i)/\mathrm{mid}(\mathbf{x}_i)|$\\\hline
        fd18 & 1.2159 & 0.3669 & 3 & 5.5511e-16 \\
        cz40948 & 1.7403 & 0.6002 & 3 & 1.1102e-16 \\
        epb3 & 6.5058 & 0.8802 & 3 & 1.1102e-16\\
        Goodwin\_127 & 65.8549 & 5.2147 & 3 & 1.1102e-16\\
        ASIC\_680k & 250.5345 & 19.5646 & 4 & 1.1102e-16\\
        \bottomrule
    \end{tabular*}
\end{table}

\section{Conclusion} \label{sec:5}

In this paper, we proposed a method for verifying the accuracy of a solution of sparse linear systems with general coefficients.
The proposed method has high stability and low memory requirements, as demonstrated by numerical experiments.
We demonstrated that the proposed method can obtain an accurate and verified solution of a linear system using iterative refinement.
Then, the proposed method contributes to one of the open problems in verified numerical computations.

\appendix

\section{MATLAB codes}\label{sec:MATLABcodes}\label{sec:appA}

This appendix gives the source code for Algorithms~\ref{alg:chol}, \ref{alg:lu}, and \ref{alg:proposed}.
The INTLAB code~\ref{code:rump} verifies the positive definiteness of $A^TA-\alpha I$.
Note that to obtain the verified solution of a linear system, the source code should be edited to store the Cholesky factor and preconditioner.
\begin{lstlisting}[style=Matlab-editor,label={code:rump},caption={Verification of lower bound of the minimum singular value using Cholesky decomposition.}]
function flag = Algorithm_1(A,alp)
    A    = intval(A);
    flag = isspd(A'*A-alp*speye(size(A)));
end
\end{lstlisting}

MATLAB code~\ref{code:minamihata} computes the approximate inverse $LU$ factors $X_L$ and $X_U$ and the enclosure $\mathbf{G}\ni X_LAX_U$.
If $\|\mathbf{G}\|_{\infty}\leq \gamma <1$, the function outputs $\rho<\|X_L\|_{\infty}\|X_U\|_{\infty}/(1-\gamma),$
which is the upper bound of $\|A^{-1}\|_{\infty}$.

\begin{lstlisting}[style=Matlab-editor,label={code:minamihata},caption={Verification of $\|X_LAX_U-I\|\leq \rho<1$.}]
function [rho,L,U,p,q,s] = Algorithm_2(A)
    A = intval(A);
    s = amd(mid(A)); A = A(s,s);
    [L,U,p,q] = lu(mid(A),'vector');
    XL = inv(L);
    alp = norm(XL,inf);
    G = XL*intval(A(p,q)); 
    clear XL
    XU = inv(U); 
    beta = norm(XU,inf);
    G = G*XU-I;
    clear XU
    gam = norm(G,inf);
    if gam < 1
        rho = sup(alp*beta/(1-gam));
    else
        rho = NaN;
    end
end
\end{lstlisting}

Finally, we give the source code for the proposed method.
The private function \texttt{NormBnd} for calculating the upper limit of the spectral norm is based on~\cite{rump2023fast}.
The private function \texttt{npe\_bdiag} counts the number of positive eigenvalues of the block diagonal matrix.
Note that this function outputs the lower bound of the minimum singular value of $\widetilde A=RPAC$ when the option $\mathtt{precond}=1$ is used.

\begin{lstlisting}[style=Matlab-editor,label={code:sigmin},caption={Verification of lower bound of the minimum singular value.}]
function [rho,L,D,p,P,R,C] = verify_sigmin(A, precond, acc)
    [m,n] = size(A);
    if ~isintval(A)
        A = intval(A);
    end
    if nargin < 3
        acc = 0;
    end
    if nargin < 2
        precond = 0;
    end
    if precond
        [P,R,C] = equilibrate(mid(A));
        r = diag(R); c = diag(C); 
        r = sign(r).*(2.^round(log2(abs(r))));
        c = sign(c).*(2.^round(log2(abs(c))));
        R = diag(r); C = diag(c);
        A = R*P*A*C;
    else
        P = speye(n); R = speye(n); C = speye(n); 
    end

    mult = intvalinit('IVmult');
    if acc
        intvalinit('SharpIVmult');
    else
        intvalinit('FastIVmult');
    end
    G = [sparse(n,n), A';
         A          , sparse(m,m)];
    rho     = 0.5*abs(eigs(mid(G),1,'smallestabs'));
    rho     = intval(rho);
    G       = G+rho*speye(m+n);
    [L,D,p] = ldl(mid(G),'vector');
    npe     = npe_bdiag(D);
    if npe == min(m, n)
        nrm = NormBnd(G(p,p)-L*intval(D)*L', 1);
        rho = rho - nrm;
    else
        rho = NaN;
    end
    rho = inf(rho);
    intvalinit(mult);
end

function output = npe_bdiag(D)
    d     = diag(D);
    f     = [diag(D,1)];
    f0    = [0;f];
    f1    = [f;0];
    ff    = f0+f1;
    s     = sum(f~=0);
    T=zeros(2*s,2);
    v=1:2:2*s-1;
    w=2:2:2*s;
    T(v,1)=d(f~=0);
    T(w,2)=d(f0~=0);
    T(v,2)=f(f~=0);
    T=intval(T);
    alp = T(v,1)+T(w,2);
    beta = sqrt(alp.*alp-T(v,1).*T(w,2)+T(v,2).*T(v,2));
    S = [alp-beta,alp+beta];
    output = sum(d(ff==0)>0)+sum(S(:)>0);
end


function N = NormBnd(A,herm)
    x = ones(size(A,1),1);
    M = [1 2];
    iter = 0;
    A = mag(A);
    while (abs(diff(M)/sum(M))>.1) && (iter<10)
        iter = iter + 1;
        y = A*x;
        if ~herm, y = A'*y; end
        x = y./x;
        M = [min(x) max(x)];
        scale = max(y);
        x = max(y/scale , 1e-12);
    end
    if herm
        N = max(mag((intval(A)*x)./x));
    else
        N = max(mag(sqrt((A'*(intval(A)*x))./x)));
    end
end
\end{lstlisting}

\section*{Acknowledgements}

The author would like to thank Dr. Atsushi Minamihata from Kansai University of International Studies, Japan, for providing practical data~\cite{minamihata2019} and helpful comments. 
This work was partially supported by JSPS KAKENHI Grant Number 23H03410.
The authors thank FORTE Science Communications for English language editing.

 \bibliographystyle{elsarticle-num} 
 \bibliography{cas-refs}





\end{document}